\documentclass{article}
\usepackage{amssymb}
\usepackage{amsmath}
\usepackage{amsthm}

\usepackage[usenames,dvipsnames]{pstricks}
 \usepackage{epsfig}

\newtheorem{theorem}{Theorem}[section]
\newtheorem{lemma}[theorem]{Lemma}
\newtheorem{example}[theorem]{\bf Example}%
\newtheorem{problem}[theorem]{\bf Problem}%
{}%

\newcommand{\N}{\mathbb{N}}
\newcommand{\G}{\mathcal{G}}

\newcommand{\fd}{\mathbb{F}}

\title{Matroid invariants and counting graph homomorphisms}
\author{Andrew Goodall\thanks{Charles University, Prague, Czech Republic. Email: \texttt{andrew@iuuk.mff.cuni.cz}. Supported by the Center of Excellence-Inst for Theor. Comp. Sci., Prague, P202/12/G061, and by Project ERCCZ LL1201 Cores.} \and Guus Regts\thanks{University of Amsterdam,  Netherlands. Email: \texttt{guusregts@gmail.com}. Supported by the European Research Council under the European Union's Seventh Framework Programme (FP7/2007-2013) / ERC grant agreement n$\mbox{}^{\circ}$ 339109, and by a NWO Veni grant.}\and Llu\'is Vena\thanks{Charles University, Prague, Czech Republic.  Email: \texttt{lluis.vena@gmail.com}.  Supported by the Center of Excellence-Inst for Theor. Comp. Sci., Prague, P202/12/G061, and by Project ERCCZ LL1201 Cores.}}

\begin{document}

\maketitle
\begin{abstract}
The number of homomorphisms from a finite graph $F$ to the complete graph $K_n$ is the evaluation of the chromatic polynomial of $F$ at $n$. Suitably scaled, this is the Tutte polynomial evaluation $T(F;1-n,0)$ and an invariant of the cycle matroid of $F$. De la Harpe and Jaeger \cite{dlHJ95} asked more generally when is it the case that a graph parameter obtained from counting homomorphisms from $F$ to a fixed graph $G$  depends only on the cycle matroid of $F$. They showed that this is true when $G$ has a generously transitive automorphism group (examples include Cayley graphs on an abelian group, and  Kneser graphs).

Using tools from multilinear algebra, we prove the converse statement, thus characterizing finite graphs $G$ for which counting homomorphisms to $G$ yields a matroid invariant.  We also extend this result to finite weighted graphs $G$ (where to count homomorphisms from $F$ to $G$ includes such problems as counting nowhere-zero flows of $F$ and evaluating the partition function of an interaction model on $F$).  

\noindent \textbf{Keywords.} graph homomorphism, graph invariant, matroid invariant

\noindent \textbf{M.S.C [2010]}  Primary   05C15, 05C60, 15A72; Secondary 52C40, 05C30, 05C31
\end{abstract}

\section{Introduction}

\subsection{Graph invariants and matroid invariants}
A graph $F$ in this paper will be finite and may have multiple edges and loops, i.e., by a graph we mean a finite multigraph. The set of vertices of $F$ is denoted by $V(F)$ and the set of edges by $E(F)$. (Parallel edges appear with multiplicity.)  
The {\em cycle matroid} of a graph $F$ is the matroid whose circuits are the edge sets of circuits in $F$. 
A loop of $F$ is a circuit of size $1$, and two parallel edges form a circuit of size $2$. A bridge in $F$ is a coloop in its cycle matroid. A matroid that is the cycle matroid of some graph is a {\em graphic matroid.} 
A standard reference for matroid theory is~\cite{O11}.

 A {\em graph invariant} is a function defined on graphs with the property that it takes the same value on isomorphic graphs. A graph invariant taking values in a field such as $\mathbb R$ is also called a {\em graph parameter}. A graph invariant taking values in a polynomial ring is a {\em graph polynomial}, important examples being the chromatic polynomial and the Tutte polynomial (for which see for example~\cite{B93,B98}).
A {\em matroid invariant} is a function defined on matroids with the property that it takes the same value on isomorphic matroids. 

Matroids were introduced by Whitney in 1935 as an abstraction of the notion of independence in linear algebra and graph theory (a subset of edges of a graph being independent if it contains no cycle). Matroid theory permits the transfer of notions defined for one type of combinatorial structure to another seemingly unrelated one. There is therefore great interest in an invariant defined for a particular combinatorial structure (such as a graph) to which there is associated a matroid (the cycle matroid of a graph), as the invariant might be extended to a larger class of matroids. 

A graph invariant whose value on a graph $F$ depends only on the underlying cycle matroid of $F$ will be called a {\em cycle matroid invariant}. A cycle matroid invariant is the restriction of a matroid invariant to graphic matroids, and for this reason de la Harpe and Jaeger~\cite{dlHJ95} use the term matroid invariant for what we call a cycle matroid invariant. The reason for our terminological difference is explained in Section~\ref{sec:concl} below. 

The Tutte polynomial of a graph $F$ is a cycle matroid invariant and the chromatic polynomial of $F$ scaled by a factor dependent only on the number of connected components of $F$ is also a cycle matroid invariant (for example trees of the same size share the same chromatic polynomial). The Tutte polynomial of $F$ may be defined in terms of the rank and size of subgraphs of $F$, which are terms defined for any matroid, and this definition can be used to define the Tutte  polynomial as a matroid invariant, whose restriction to graphic matroids is the Tutte polynomial for graphs. 
This is one reason why the Tutte polynomial has played such a central role in combinatorics: matroids encompass a great diversity of combinatorial structures and the richly developed theory of graphs has illuminated through the lens of the Tutte polynomial such areas as knot theory and statistical physics (see for example~\cite{W93}). 

In this paper we characterize a set of cycle matroid invariants defined by counting graph homomorphisms, a set that includes the chromatic polynomial and the Tutte polynomial. In contrast to the chromatic polynomial and the Tutte polynomial, however, the extension of these cycle matroid invariants to matroid invariants is not known in general, although for some of  them there are naturally defined extensions to matroids with additional structure (such as oriented matroids). We briefly discuss this question further in Section~\ref{sec:concl}. 



 \subsection{Graph homomorphisms and the question of de la Harpe and Jaeger}

For $n\in \N$, a field $\fd$ of characteristic zero, a vector $a\in (\fd^*)^{n}$ and a symmetric matrix $B\in \fd^{n\times n}$, we let $G(a,B)$ denote the vertex- and edge-weighted graph with vertex set $\{1,2,\dots, n\}$ in which vertex $i$ has weight $a_i$ and edge $ij$ has weight $B_{i,j}$. 
When $B$ has nonnegative integer entries and $a_i=1$ for each $i$ the weighted graph $G(a,B)$ can be viewed as a multigraph  that has  adjacency matrix $B$, the entry $B_{i,j}$ giving the multiplicity of the edge $ij$. 

The graph parameter defined by
\begin{equation}\label{eq:def_hom}
F\mapsto \hom(F,G(a,B))=\sum_{\phi:V(F)\to [n]}\:\:\prod_{v\in V(F)}a_{\phi(v)}\:\cdot\,\prod_{uv\in E(F)}B_{\phi(u),\phi(v)},
\end{equation}
where edges $uv\in E(F)$ are taken with multiplicity in the product, includes many important graph invariants, such as the chromatic polynomial evaluated at $n$, the flow polynomial at $n$ and evaluations of the Tutte polynomial (see Example~\ref{ex:Tutte} below), the independence polynomial, and the partition functions of other interaction models in statistical physics. See for example the survey~\cite{BCLSV06} and the recent book by Lov\'asz~\cite{L12} for the parameter ${\rm hom}(\,\cdot\,,G(a,B))$ and~\cite{B77} for partition functions of interaction models.

If $a_i=1$ for each $i$ and $B$ with entries $B_{i,j}\in\{0,1\}$ is the adjacency matrix of a graph $G$, then $\hom(F,G(a,B))$ is equal to the number of homomorphisms from $F$ to $G$. The definition of $\hom(\cdot,G(a,B))$ given in~\eqref{eq:def_hom} generalizes homomorphism counting between graphs to a weighted target graph. 

The number of connected components of a graph $F$ is denoted by $c(F)$ and its rank by $r(F)$ (equal to $|V(F)|-c(F)$). For $A\subseteq E(F)$, the rank of the subgraph $(V(F), A)$ is denoted by $r(A)$.  
\begin{example}\label{ex:Tutte} Let $G=G(\mathbf 1, (y-1)I+J)$, where $\mathbf 1$ is the $n\times 1$ vector with each entry equal to $1$, $I$ is the $n\times n$ identity matrix and $J=\mathbf 1\mathbf 1^T$. Then (see for example~\cite{dlHJ95})
$$\hom(F,G)=n^{c(F)}(y-1)^{r(F)}T\left(F;\frac{y-1+n}{y-1},y\right),$$
where $T(F;x,y)$ is the Tutte polynomial of $F$, defined by
$$T(F;x,y)=\sum_{A\subseteq E(F)}(x-1)^{r(F)-r(A)}(y-1)^{|A|-r(A)}.$$
In particular, when $y=0$ we have $G=K_n$ and $$\hom(F,K_n)=(-1)^{r(F)}n^{c(F)}T(F;1-n,0)=\chi(F;n)$$ is the chromatic polynomial of $F$ evaluated at $n$, counting the number of proper $n$-colourings of $F$. When $y=1-n$, where now $G$ is the complete graph with a loop of weight $1-n$ on each vertex, 
$$\hom(F,G)=(-1)^{r(F)}n^{|V(F)|}T(F;0,1-n)=(-1)^{|E(F)|}n^{|V(F)|}\phi(F;n),$$ 
where $\phi(F;n)$ is the flow polynomial of $F$ evaluated at $n$, counting the number of nowhere-zero $\mathbb Z_n$-flows of $F$.
\end{example}

The Tutte polynomial of a graph $F$ is an invariant of the cycle matroid of $F$. Example~\ref{ex:Tutte} motivates the following question, posed by de la Harpe and Jaeger
in~\cite{dlHJ95}:
for which (unweighted) graphs $G$ is the graph invariant
\[
h(\,\cdot\,,G): F\mapsto\frac{\hom(F,G)}{|V(G)|^{c(F)}}
\]
dependent only on the cycle matroid of $F$?

In \cite{dlHJ95} it is shown that if $G$ has a generously transitive automorphism group then $h(\cdot,G)$ is a cycle matroid invariant. 
(The automorphism group $\Gamma$ of a graph $G$ is \emph{generously transitive} if for each $u,v\in V(G)$ there exists $\gamma\in \Gamma$ such that $\gamma(u)=v$ and $\gamma(v)=u$. If a group action is generously transitive then it is also transitive.)
We extend their result here to edge-weighted graphs $G$ and moreover show that the sufficient condition is also necessary, thus answering their question.
\begin{theorem}\label{thm:matroid invariant}
Let $G=G(\mathbf 1,B)$ be an edge-weighted graph.
Then the graph invariant $h(\,\cdot\,, G)$ is a cycle matroid invariant if and only if $G$ has a generously transitive automorphism group. 
\end{theorem}

 
To prove Theorem~\ref{thm:matroid invariant} we shall use graph algebras, as introduced by Freedman, Lov\'asz and Schrijver~\cite{FLS07}.

\section{Graph algebras}
Let $\G_k$ denote the set of $k$-labelled graphs, i.e., elements of $\G_k$ are graphs in which $k$ distinct vertices are labelled with the integers $1,\ldots,k$.
For a given $k$-labelled graph it will be convenient to identify the set $[k]:=\{1,\ldots,k\}$ with the subset of $k$ vertices that receive a label.
Any graph invariant can be extended to $k$-labelled graphs  by forgetting the labels. 

We make $\G_k$ into a semigroup by defining for $G_1,G_2\in \G_k$ the product $G_1\cdot G_2$ to be the $k$-labelled graph obtained from the disjoint union of $G_1$ and $G_2$ by identifying vertices that have the same label. (This operation applied to simple graphs may produce multiple edges; we recall that graphs for us are allowed to have multiple edges and loops.)
Given a field $\fd$ of characteristic zero,  let $\fd \G_k$ be the semigroup algebra of $(\G_k,\cdot)$, whose elements are finite formal $\fd$-linear combinations of $k$-labelled graphs. See for example~\cite{LS07} for more on the semigroup algebra $\fd\G_k$.

For a graph parameter $f:\G_0\to \fd$ and $k\in \N$ we define 
\[
\mathcal I(f)_k:=\{\gamma\in \fd \G_k\mid f(\gamma\cdot F)=0 \text{ for all } F\in \G_k\},
\]
where the labels of $\gamma\cdot F$ are forgotten in evaluating $f(\gamma\cdot F)$. The set $\mathcal I(f)_k$ is an ideal in $\fd\G_k$. 

Now we fix a weighted graph $G(a,B)$ on $n$ vertices, where  $a\in (\fd^*)^{n}$ and $B$ is a symmetric matrix in $\fd^{n\times n}$. 
Let $V:=\fd^n$ and let $e_1,\ldots, e_n$ be the standard basis for $V$. 
For $k\in \N$, let $p_{a,B}:\fd\G_k\to V^{\otimes k}$ be the linear map defined for $F\in \G_k$ by
\[
F\mapsto \sum_{\phi:V(F)\to [n]}\left(\prod_{v\in V(F)\setminus [k]}a_{\phi(v)}\cdot \prod_{uv\in E(F)}B_{\phi(u),\phi(v)}\right) e_{\phi(1)}\otimes \cdots \otimes e_{\phi(k)}.
\]
For $k=0$ the map $p_{a,B}$ is the graph parameter $\hom(\,\cdot\,,G(a,B))$.
If $\sum_{i=1}^na_i\neq 0$, the graph parameter $h(\,\cdot\,,G)$ can be generalized to weighted graphs by
\[
h(F,G(a,B)):=\frac{\hom(F,G(a,B))}{\left[\sum_{i=1}^{n}a_i\right]^{c(F)}}.
\]

The automorphism group $\Gamma(a,B)$ of $G(a,B)$ is the subgroup of permutations of the vertex set of $G(a,B)$ that preserve all vertex and edge weights of $G(a,B)$.
The group $\Gamma(a,B)$ acts naturally on $V^{\otimes k}$ for any $k\in \N$. 
One way to realize this action is to consider the basis vectors $e_i$ as the vertices of the graph $G(a,B)$, for which $\gamma\in \Gamma(a,B)$ sends $e_{i_1}\otimes \cdots \otimes e_{i_k}$ to $\gamma(e_{i_1})\otimes \cdots\otimes \gamma(e_{i_k})$. 
Here we use the fact that the tensors $e_{i_1}\otimes \cdots\otimes e_{i_k}$ form a basis for $V^{\otimes k}$.

A weighted graph $G(a,B)$ is \emph{twin-free} if no two rows of $B$ are equal.
To prove our main result, we will need a characterization of the image of $p_{a,B}$ for twin-free graphs in terms of the tensors invariant under $\Gamma(a,B)$, building on and generalizing a result of Lov\'asz~\cite[Theorem 2.2]{L06} (see also~\cite{L12}) and a result of the second author~\cite[Theorem 6.15]{R13}.
\begin{theorem}\label{thm:tensor}
Suppose that $G(a,B)$ is twin-free. 
Then $(\ker p_{a,B})\cap \fd\G_k=\mathcal I(p_{a,B})_k$ and $p_{a,B}(\fd\G_k)=(V^{\otimes k})^{\Gamma(a,B)}$, the space of tensors that are invariant under the action of $\Gamma(a,B)$.
\end{theorem}
\begin{proof}
We start with the second statement.
Let us denote the image of $\mathbb{F}\G_k$ under $p_{a,B}$ by $\mathcal A_k''$ and the space of $\Gamma$-invariant tensors in $V^{\otimes k}$ by $\mathcal A_k'$, where $\Gamma:=\Gamma(a,B)$.

It is easy to see that $\mathcal A_k''\subseteq \mathcal A_k'$.
To show the reverse inclusion we use an idea from \cite{R13} that allows us to reduce the argument to that given in the proof of Lov\'asz's result~\cite{L06}.
Fix $k\in \N$. For $\phi:[k]\to [n]$ we define $e_{\phi}:=e_{\phi(1)}\otimes \cdots \otimes e_{\phi(k)}$ and 
for $\phi,\psi:[k]\to [n]$ we define $e_{\phi}* e_{\psi}:=e_{\phi}\delta_{\phi,\psi}$ and extending this bilinearly to $V^{\otimes k}$. (Here $\delta$ is the Kronecker delta.) This makes $V^{\otimes k}$ into an algebra. 
Both $\mathcal A'_k$ and $\mathcal A''_k$ are subalgebras of $V^{\otimes k}$. 

Put an equivalence relation $\sim$ on $\{\phi:[k]\to [n]\}$ by declaring that $\phi\sim \psi$ if and only if for each $x=\sum_{\kappa:[k]\to [n]}x_\kappa e_\kappa\in \mathcal A_k''$ we have $x_{\phi}=x_{\psi}$.
Let $E_1,\ldots E_t$ be the equivalence classes of this relation. 
Note that $\sum_{\phi:[k]\to [n]}e_{\phi}\in \mathcal A_k''$, as it is the image of the disjoint union of $k$ labelled vertices.
Then
\begin{equation}
\Phi_i:=\sum_{\phi\in E_i}e_{\phi}\in \mathcal A_k'' \text{ for } i=1,\ldots,t.	\label{eq:eq classes}
\end{equation}
To see \eqref{eq:eq classes} let $z=\sum_{i=1}^t z_i \Phi_i\in \mathcal A''_k$ be such that the number of distinct coefficients $z_i$ is maximal. 
Then all the $z_i$ are distinct. For if this were not true then without loss of generality $z_1=z_2$. 
By definition of the equivalence relation~$\sim$, there exists $y=\sum_{i=1}^ty_i\Phi_i$ such that $y_1\neq y_2$.
Now pick a nonzero $\lambda$ with the property that $\lambda z_i+y_i\neq \lambda z_j+y_j$ for all $i,j$ such that $z_i\neq z_j$.
Then $\lambda z+y$ contains more distinct coefficients than $z$, which is a contradiction. 
Now choose interpolating polynomials $q_i$ such that $q_i(z_j)=\delta_{i,j}$, cf. \cite[Lemma 2.9]{CLS04}.
Then since $\mathcal A''_k$ is an algebra we have $q_i(z)=\Phi_i\in \mathcal A''_k$, proving \eqref{eq:eq classes}.

By Lemma 2.4 from~\cite{L06}\footnote{Even though Lov\'asz \cite{L06} works over $\mathbb R$ and assumes $a_i>0$ for all $i$, it is easy to check that the proof of his Lemma 2.4 remains valid in our setting.} we know that the $E_i$ are precisely the orbits of $\Gamma$ acting on $\{\phi:[k] \to [n]\}$.
From this we conclude that $\mathcal A'_k=\mathcal A''_k$, which establishes the second statement of the theorem.

As for the first statement, we clearly have that $(\ker p_{a,B})\cap \fd\G_k \subseteq \mathcal I(p_{a,B})_k$.
To see the reverse inclusion, consider the unique bilinear form on $V^{\otimes k}$ defined for $\phi,\psi:[k]\to [n]$ by  $(e_{\phi},e_{\psi}):=\delta_{\phi,\psi}$.
Then $(p_{a,B}(F_1),p_{a,B}(F_2))=\hom(F_1\cdot F_2,G)$ for $F_1,F_2\in \G_k$.
Now let $x\in \mathcal I(p_{a,B})_k$ and suppose that $x\notin  (\ker p_{a,B})\cap \fd\G_k$.
Then, since the form $(\cdot,\cdot)$ is nondegenerate on $V^{\otimes k}$, there exists $v\in V^{\otimes k}$ such that $(p_{a,B}(x),v)=1$. 
As $p_{a,B}(x)\in \mathcal A''_k=\mathcal A'_k$, we know that $p_{a,B}(x)$ is invariant under $\Gamma$.
So we have
\begin{equation}\label{eq:form=1}
1=(p_{a,B}(x),v)=\frac{1}{|\Gamma|} \sum_{\gamma \in \Gamma}(\gamma p_{a,B}(x),v)=(p_{a,B}(x),\frac{1}{|\Gamma|} \sum_{\gamma \in \Gamma}\gamma v).
\end{equation}
Now $\frac{1}{|\Gamma|} \sum_{\gamma \in \Gamma}\gamma v\in \mathcal A'_k=\mathcal A''_k$. Thus there exists $y\in \fd\G_k$ such that $v=p_{a,B}(y)$.
Writing $y=\sum_{i=1}^l y_i F_i$ for certain $F_i\in \G_k$ and $y_i\in \fd$, we now have that $(p_{a,B}(x),p_{a,B}(y))=\sum_{i=1}^l y_i \hom(x\cdot F_i,G)=0$, contradicting \eqref{eq:form=1}.
This finishes the proof. 
\end{proof}

We now collect some important consequences of this result.
\begin{lemma}\label{lem:trans}
Suppose that $G:=G(a,B)$ is twin-free and $\sum_{i=1}^n a_i\neq 0$.
Then $h(F_1,G)h(F_2,G)=h(F_1\cdot F_2,G)$ for all $F_1,F_2\in \G_1$ if and only if $\Gamma:=\Gamma(a,B)$ acts transitively on $G$.
\end{lemma}
\begin{proof}
We first prove that if $h(\,\cdot\,,G)$ is multiplicative over the the product of elements of $\G_1$ then the automorphism group of $G$ must act transitively. 
Let $F_1,F_2\in \G_1$ and let $F=F_1\cdot F_2$. Let $\bullet \in \G_1$ denote the graph with one vertex and no edges.
Note first that $c(F)=c(F_1)+c(F_2)-1$. 
By assumption $h(F,G)=h(F_1,G)h(F_2,G)$, so 
$\left(\sum_{i=1}^n a_i\right)\hom(F,G)=\hom(F_1,G)\hom(F_2,G)$ and if 
$\hom(F_1,G)=0$, then $F_1\in \mathcal{I}(p_{a,B})_1$ as $p_{a,B}$ acts as $\hom(\cdot,G)$ when the labels on the vertices are forgotten. 
We have $\hom(\bullet,G)\neq 0$ and $\bullet \cdot \bullet =\bullet$, so that $\bullet\notin \mathcal I(p_{a,B})_1$ and $\dim(\fd\G_1/\mathcal I(p_{a,B})_1)\geq 1$. Writing $F_1$ as 
$$F_1=\left[F_1-\frac{\hom(F_1,G)}{\hom(\bullet,G)} \bullet\right]+\frac{\hom(F_1,G)}{\hom(\bullet,G)}\bullet$$ (a sum of an element in $\mathcal{I}(p_{a,B})_1$ and a scalar multiple of $\bullet$), we conclude that 
$\dim(\fd\G_1/\mathcal I(p_{a,B})_1)=1$.
Hence by Theorem~\ref{thm:tensor} the space $V^\Gamma$ is one-dimensional, which implies that $\Gamma$ acts transitively on~$G$.

Conversely, suppose that $\Gamma$ acts transitively on $G$. 
For each homomorphism $F_1\cdot F_2\to G$ there are $n$ distinct homomorphisms 
 from the disjoint union of $F_1$ and $F_2$ to $G$; the image of the labelled vertex of $F_1$ is the same as that of $F_1\cdot F_2$, leaving  $n$ choices still for the labelled vertex of $F_2$ using the transitive action of $\Gamma$. (Note that we did not use the assumption that $G$ is twin-free here.)
\end{proof}
For $F\in \G_2$ define $F^T$ to be the $2$-labelled graph arising from $F$ by switching its labels and extend this operation linearly to $\fd \G_2$.

\begin{lemma}\label{lem:gen trans}
Suppose that $G:=G(a,B)$ is twin-free and $\sum_{i=1}^n a_i\neq 0$.
Then $h(F_1\cdot F_2,G)=h(F_1^T\cdot F_2,G)$ for all $F_1,F_2\in \G_2$ if and only if $\Gamma:=\Gamma(a,B)$ acts generously transitively on $G$.
\end{lemma}
\begin{proof}

Define a \emph{transposition} on $V^{\otimes 2}$ by $(e_i\otimes e_j)^T:=e_j\otimes e_i$ and extending this linearly to $V^{\otimes2}$.

Let us begin with the assumption that $\hom(F_1\cdot F_2,G)=\hom(F_1^T\cdot F_2,G)$ for all $F_1,F_2\in \G_2$. 
So $F_1-F_1^T\in \mathcal I(p_{a,B})_2$ for each $F_1\in \G_2$.
By Theorem \ref{thm:tensor} this implies that $p_{a,B}(F_1)=p_{a,B}(F_1^T)=p_{a,B}(F_1)^T$ for each $F_1\in \fd\G_2$.
Also, by Theorem \ref{thm:tensor}, the image of $p_{a,B}$ is equal to $(V^{\otimes 2})^{\Gamma}$.

Now fix $i,j\in [n]$ and consider $v:=\sum_{\gamma\in \Gamma}\gamma(e_i)\otimes \gamma(e_j)$, the $\Gamma$-orbit of $e_i\otimes e_j$.
Then $v$ is clearly invariant under $\Gamma$ and hence there exists $F\in \fd\G_2$ such that $p_{a,B}(F)=v=v^T=p_{a,B}(F^T)$, which implies that $e_{j}\otimes e_{i}=\gamma(e_i)\otimes \gamma(e_j)$ for some $\gamma\in \Gamma$.
This exactly means that $\Gamma$ acts generously transitively on $G$.

Conversely, if $\Gamma$ acts generously transitively on $G$, this implies that $e_i\otimes e_j$ and $e_j\otimes e_i$ are in the same $\Gamma$-orbit for each $i,j\in [n]$.
From this it follows that $v^T=v$ for all $v\in (V^{\otimes 2})^\Gamma$ hence for each $F_1\in \G_2$ we have that $F_1-F_1^T\in \mathcal I(p_{a,B})_2$. 
In other words, $\hom(F_1\cdot F_2,G)=\hom(F_1^T\cdot F_2,G)$ for all $F_2\in \G_2$.
\end{proof}

\section{Proof of Theorem \ref{thm:matroid invariant}}
We need a result about automorphism groups of graphs with twins before giving our proof.

Recall that for a weighted graph $G(a,B)$ on $[n]$, vertices $i,j\in [n]$ are \emph{twins} if the $i$th row and $j$th row of $B$ are the same (regardless of the values $a_i$ and $a_j$).
This defines an equivalence relation on $[n]$. Let $C_1,\ldots,C_m$ be the equivalence classes. 
Define a weighted graph $G(a',B')$ on $[m]$ by letting $a'_i:=\sum_{j\in C_i}a_j$ for $i\in [m]$ and in which $B'$ is obtained from $B$ by removing all but one of the rows and columns indexed by $C_i$, for each $i\in [m]$. (In case $a_i'=0$ for some $i$ we just remove $i$ from $[m]$ and also the corresponding row and column from $B'$.)
We shall call $G(a',B')$ the \emph{twin-reduced graph}.
It is not difficult to see that $\hom(F,G(a,B))=\hom(F,G(a',B'))$ for all graphs $F$, cf.~\cite{L06}.
We denote by $\mathbf 1$ the vector in $\N^n$ with each entry equal to $1$ (for any $n$).
\begin{lemma} \label{lem:twin groups}
Let $G:=G({\bf 1},B)$ be a weighted graph with twin-reduced graph $G':=G(a',B')$.
Then $\Gamma:=\Gamma({\bf1}, B)$ acts (generously) transitively on $G({\bf 1},B)$ if and only if $\Gamma':=\Gamma(a',B')$ acts (generously) transitively on $G(a',B')$.
\end{lemma}
\begin{proof}
If $\Gamma'$ acts transitively on $G'$ then $a'$ is equal to $m{\bf 1}$ for some $m\in \N$.
So all equivalence classes of $G$ are of equal size and permuting these classes using $\Gamma'$ yields an automorphism of $G$.
Since permuting elements inside an equivalence class is an automorphism of $G$, this implies that if $\Gamma'$ acts (generously) transitively on $G'$ then so does $\Gamma$ on $G$.

Conversely, suppose that $\Gamma$ acts (generously) transitively on $G$.
Then $\Gamma$ preserves the equivalence relation, i.e., if $i$ and $j$ are twins and $\gamma\in \Gamma$, then $\gamma(i)$ and $\gamma(j)$ are twins.
Indeed, for each $k\in [n]$, $B_{i,k}=B_{j,k}=B_{\gamma(j),\gamma(k)}=B_{\gamma(i),\gamma(k)}$ and by transitivity $\{\gamma(k)\mid k\in [n]\}=[n]$.
From this it follows that $\Gamma'$ acts (generously) transitively on $G'$, finishing the proof.
\end{proof}

We are now able to give a proof of Theorem \ref{thm:matroid invariant}
\begin{proof}[Proof of Theorem \ref{thm:matroid invariant}]
Let $G:=G({\bf 1},B)$ be a weighted graph such that $h(\cdot,G)$ defines a cycle matroid invariant.
We need to show that $\Gamma:=\Gamma({\bf 1},B)$ acts generously transitively on $G$.
By Lemma \ref{lem:twin groups} we may assume that $G$ is twin-free.
For $F_1,F_2\in \G_2$, the edges of $F_1\cdot F_2$ are in one-to-one correspondence with the edges of $F_1^T\cdot F_2$ such that circuits are maintained. Hence the cycle matroids of $F_1\cdot F_2$ and $F_1^T\cdot F_2$ are isomorphic.
By the assumption that $h(\,\cdot\,,G)$ is a cycle matroid invariant, this implies that $h(F_1\cdot F_2,G)=h(F_1^T\cdot F_2,G)$. As this holds for all $F_1,F_2\in \G_2$, Lemma \ref{lem:gen trans} implies that $\Gamma$ acts generously transitively on $G$.

For the converse we follow the line of proof given for Proposition 7 in~\cite{dlHJ95}. Suppose that $\Gamma:=\Gamma({\bf 1},G)$ acts generously transitively on $G:=G({\bf1 },B)$.
Again we may assume that $G$ is twin-free. 
Let us denote the cycle matroid of a graph $F$ by $M(F)$.
Whitney's $2$-isomorphism theorem \cite{W33} characterizes when two graphs $F,F'$ have the same cycle matroid in terms of the following three operations:
\begin{itemize}
\item [(i)] for $F_1,F_2\in \G_2$, the operation of transforming $F_1\cdot F_2$ into $F_1^T\cdot F_2$ (this operation is known as a \emph{Whitney flip});
\item [(ii)] for $F_1,F_2\in \G_1$, the operation of transforming the disjoint union of $F_1$ and $F_2$ into $F_1\cdot F_2$;
\item[(iii)] for $F_1,F_2\in \G_1$  the operation of transforming of $F_1\cdot F_2$ into the disjoint union of $F_1$ and $F_2$.
\end{itemize}
Then Whitney's $2$-isomorphism theorem states that $M(F)\cong M(F')$ if and only if $F$ can be obtained from $F'$ by applying a sequence of the three operations above.
Whitney's result together with Lemmas~\ref{lem:trans} and~\ref{lem:gen trans} clearly imply that if $M(F)=M(F')$ then $h(F,G)=h(F',G)$ when $G$ has a generously transitive automorphism group.
\end{proof}

\section{Concluding remarks}\label{sec:concl}

We have used the term {\em cycle matroid invariant} for a 
graph invariant whose value on a graph $F$ depends only on the cycle matroid of $F$. This is in order to avoid the confusion of a graph invariant which is an invariant of the underlying cycle matroid with the more widely defined matroid invariant of which it is the restriction to graphic matroids.
Indeed, for a given cycle matroid invariant it is not always known to which matroid invariant it extends. For example, on page 608 of~\cite{dlHJ95} the question is left open of identifying the matroid invariant that restricted to graphic matroids gives the cycle matroid invariant obtained by counting homomorphisms from a graph $F$ to a generalized Johnson graph or Grassmann graph. 

We elaborate this remark in the context of Cayley graphs. 

Let $A$ be an abelian group and $B$ a subset of $A$ such that $-B=B$. Let ${\rm Cayley}(A,B)$ denote the associated Cayley graph in which vertices $u$ and $v$ are joined by an edge when $v-u\in B$. The graph ${\rm Cayley}(A,B)$ has a generously transitive automorphism group (in additive notation, the graph automorphism $x\mapsto u+v-x$ swaps vertices $u$ and $v$). By Theorem~\ref{thm:matroid invariant} it follows that the number of $B$-tensions of a graph $F$ is an invariant of the cycle matroid of $F$, as also shown in ~\cite[Proposition 6]{dlHJ95}. (An $A$-tension of a graph $F$ with a fixed orientation of its edges is a function $E(F)\to A$ with the property that the sum of values on forward edges when traversing a circuit is equal to the sum of values on backward edges. A  $B$-tension of $F$ is an $A$-tension that takes values only in $B$.) A $B$-tension of a graph $F$ is defined in terms of the oriented cycle matroid of $F$, in which circuits are signed according to edge orientations in a fixed but arbitrary orientation of $F$. The number of $B$-tensions of a graph $F$ is however independent of the choice of orientation, and for this reason the number of $B$-tensions is in this case properly a cycle matroid invariant.
By using the definition of signed circuits in an oriented matroid, $B$-tensions can be defined more generally for orientable~matroids. 

For example, the graph ${\rm Cayley}(\mathbb Z_n,\mathbb Z_n\setminus\{0\})$ is isomorphic to $K_{n}$ and $\mathbb Z_n\setminus\{0\}$-tensions of $F$ are nowhere-zero $\mathbb Z_n$-tensions, to each of which correspond $n^{c(F)}$ proper vertex $n$-colourings of $F$. Our graph invariant $h(\,\cdot\,,G)$ here for $G=K_n$ is given by $h(F,K_n)=(-1)^{r(F)}T(F;1-n,0)$. 
However, it is not the case for an oriented matroid $F$ in general that $(-1)^{r(F)}T(F;1-n,0)$ equals the number of nowhere-zero $\mathbb Z_n$-tensions of $F$: it is only on the class of graphic matroids that the matroid invariant $(-1)^{r(F)}T(F;1-n,0)$ has this interpretation.
Therefore the cycle matroid invariant $h(F,K_n)$ defined for graphs $F$ can be extended in two ways to a larger class of matroids: to the class of all matroids by setting $h(F,K_n)=(-1)^{r(F)}T(F;1-n,0)$, or to the class of orientable matroids $F$ (together with an orientation class -- see observation (5) at the end of~\cite{GTZ98}) by setting $h(F;K_n)$ equal to the number of nowhere-zero $\mathbb Z_n$-tensions of $F$ (with an orientation chosen arbitrarily from the given orientation class of~$F$).

Another reason thus emerges for using the term {\em cycle matroid invariant} rather than matroid invariant  for a graph invariant that depends only on the underlying cycle matroid: 
the definition of a given invariant for graphs may not extend to all finite matroids, as it may depend on some extra structure attached to graphic matroids (such as an orientation) even though for graphic matroids its value may be independent of this additional structure (as for nowhere-zero tensions). 

In other words, we have seen how a cycle matroid invariant may also be the restriction of an oriented matroid invariant (which has the property that it is a matroid invariant when restricted to graphic matroids), as well as being the restriction of a matroid invariant (different from the oriented matroid invariant).

\subsection*{Acknowledgement}
The authors are grateful to the referee for a careful reading of the manuscript.

\end{document}